\documentclass{amsart}
\usepackage{amsmath,amssymb,amsthm,url,comment}
\usepackage{graphicx}

\newtheorem{theorem}{Theorem}[section]  

\newtheorem{conjecture}{Conjecture} 
 
\newtheorem{corollary}[theorem]{Corollary}
\theoremstyle{definition}

\newtheorem*{remark*}{Remark}
\newcommand{\Z}{\mathbb{Z}}

\newcommand{\Q}{\mathbb{Q}}

\newcommand{\linef}{
\raisebox{-4mm}{
\begin{picture}(20,30)
\put(2,0){\line(0,1){30}}
\put(5,22){\scriptsize $f(\hbar)$}
\end{picture} 
}} 

\newcommand{\linefa}{
\raisebox{-4mm}{
\begin{picture}(10,30)
\put(2,0){\line(0,1){30}}
\end{picture} 
}} 

\newcommand{\linefb}{
\raisebox{-4mm}{
\begin{picture}(10,30)
\put(2,0){\line(0,1){30}}
\put(2,15){\line(1,0){10}}
\end{picture} 
}} 
\newcommand{\linefc}{
\raisebox{-4mm}{
\begin{picture}(10,30)
\put(2,0){\line(0,1){30}}
\put(2,12){\line(1,0){10}}
\put(2,18){\line(1,0){10}}
\end{picture} 
}} 

\newcommand{\linefd}{
\raisebox{-4mm}{
\begin{picture}(10,30)
\put(2,0){\line(0,1){30}}
\put(2,9){\line(1,0){10}}
\put(2,15){\line(1,0){10}}
\put(2,21){\line(1,0){10}}
\end{picture} 
}} 

\newcommand{\oneloop}{
\raisebox{-7mm}{
\begin{picture}(140,60)
\put(70,22){\oval(140,20)}
\put(5,40){\footnotesize $\frac{1}{2}\log\left(\frac{\sinh(\hbar/2)}{\hbar/2}\right) -\frac{1}{2}\log(\Delta_K(e^{\hbar}))$}
\end{picture}
}
} 

\newcommand{\twoloop}{
\raisebox{-8mm}{
\begin{picture}(80,60)
\put(40,25){\oval(80,50)}
\put(0,25){\line(1,0){80}}
\put(10,54){\footnotesize $p_{i,1}(e^{\hbar})\slash \Delta_K(e^{\hbar})$}
\put(10,29){\footnotesize $p_{i,2}(e^{\hbar})\slash \Delta_K(e^{\hbar})$}
\put(10,6){\footnotesize $p_{i,3}(e^{\hbar})\slash \Delta_K(e^{\hbar})$}
\end{picture}
}
}

 \makeatletter
    
    \@addtoreset{equation}{section}
  \makeatother
 
\begin{document}

\title[Cosmetic crossing conjecture for genus one knot]{Cosmetic crossing conjecture for genus one knots with non-trivial Alexander polynomial}
\author{Tetsuya Ito}


\begin{abstract}
We prove the cosmetic crossing conjecture for genus one knots with non-trivial Alexander polynomial. We also prove the conjecture for genus one knots with trivial Alexander polynomial, under some additional assumptions.
\end{abstract}

\maketitle

\section{Introduction}

A \emph{cosmetic crossing} is a non-nugatory crossing such that the crossing change at the crossing preserves the knot. A cosmetic crossing conjecture \cite[Problem 1.58]{ki} asserts there are no such crossings.

\begin{conjecture}[Cosmetic crossing conjecture]
An oriented knot $K$ in $S^{3}$ does not have cosmetic crossings.
\end{conjecture}

Here a crossing $c$ of a knot diagram $D$ is \emph{nugatory} if there is a circle $C$ on the projection plane that transverse to the diagram $D$ only at $c$. Obviously the crossing change at a nugatory crossing always preserves the knot, so the cosmetic crossing conjecture can be rephrased that when a crossing change at a crossing $c$ preserves the knot, then $c$ is nugatory.

In \cite{bfkp} Balm-Friedl-Kalfagianni-Powell proved the following constraints for genus one knots to admit a cosmetic crossing.

\begin{theorem}\cite[Theorem 1.1, Theorem 5.1]{bfkp}
\label{theorem:genus-one}
Let $K$ be a genus one knot that admits a cosmetic crossing. Then $K$ has the following properties.
\begin{itemize}
\item $K$ is algebraically slice.
\item For the double branched covering $\Sigma_2(K)$ of $K$, $H_1(\Sigma_2(K);\Z)$ is finite cyclic.
\item If $K$ has a unique genus one Seifert surface, $\Delta_K(t)=1$.
\end{itemize}
\end{theorem}

In this paper, by using the 2-loop part of the Kontsevich invariant, we prove the cosmetic crossing conjecture for genus one knot with non-trivial Alexander polynomial.

\begin{theorem}
\label{theorem:main}
Let $K$ be a genus one knot. If $\Delta_K(t)\neq 1$, then $K$ satisfies the cosmetic crossing conjecture.
\end{theorem}

For genus one knot $K$ with $\Delta_K(t)=1$ we get an additional constraint for $K$ to admit a cosmetic crossing. Let $\lambda$ be the Casson invariant of integral homology spheres and let $w_3(K)=\frac{1}{36}V'''_K(1)+\frac{1}{12}V''_K(1)$ be the primitive integer-valued degree $3$ finite type invariant of $K$. Here $V_K(t)$ is the Jones polynomial of $K$. 

\begin{theorem}
\label{theorem:main-trivialA}
Let $K$ be a genus one knot with $\Delta_K(t)=1$. If $\lambda(\Sigma_2(K))-2w_3(K) \not \equiv 0 \pmod{16}$, then $K$ satisfies the cosmetic crossing conjecture. 
\end{theorem}

The cosmetic crossing conjecture has been confirmed for several cases; 2-bridge knots \cite{tor}, fibered knots \cite{kal}, knots whose double branched coverings  are L-spaces with square-free 1st homology \cite{lm}, and some satellite knots \cite{bk}. Except the last satellite cases and the unknot, all the knots mentioned so far, including knots treated in Theorem \ref{theorem:genus-one}, has non-trivial Alexander polynomial.

Theorem \ref{theorem:main-trivialA} gives examples of non-satellite knots with trivial Alexander polynomial satisfying the cosmetic crossing conjecture. Let $K=P(p,q,r)$ be the pretzel knot for odd $p,q,r$. Obviously, as long as $K$ is non-trivial, $g(K)=1$. The Alexander polynomial of $K$ is  
\[ \Delta_K(t) = \frac{pq+qr+rp+1}{4}t + \frac{-2pq-2qr-2rp+1}{2} + \frac{pq+qr+rp+1}{4}t^{-1}.\]
Hence, for example, the pretzel knot $P(4k+1,4k+3,-(2k+1))$ has the trivial Alexander polynomial. 
\begin{corollary}
\label{corollary:pretzel}
If $k\equiv 1,2 \pmod 4$, the pretzel knot $P(4k+1,4k+3,-(2k+1))$ satisfies the cosmetic crossing conjecture.
\end{corollary}

\section{Cosmetic crossing of genus one knot and Seifert surface}

We review an argument of \cite[Section 2, Section 3]{bfkp} that relates a cosmetic crossing change and Seifert matrix. 

A \emph{crossing disk} $D$ of an oriented knot $K$ is an embedded disk having exactly one positive and one negative crossing with $K$. A crossing change can be seen as $\varepsilon = \pm 1$ Dehn surgery on $\partial D$ for an appropriate crossing disk $D$, and the crossing is nugatory if and only if $\partial D$ bounds an embedded disk in $S^{3} \setminus K$.

Assume that $K$ admits a cosmetic crossing with the crossing disk $D$. Then as is discussed in \cite[Section 2]{bfkp}, there is a minimum genus Seifert surface $S$ of $K$ such that $\alpha:=D \cap S$ is a properly embedded, essential arc in $S$.

If $g(S)=1$, such an arc $\alpha$ is non-separating. We take simple closed curves $a_x,a_y$ of $S$ so that
\begin{itemize}
\item $a_x$ intersects $\alpha$ exactly once.
\item $a_x$ and $a_y$ form a symplectic basis of $H_1(S;\Z)$.
\end{itemize}
Then we view $K=\partial S$ as a neighborhood of $a_x \cup a_y$ and express $K$ by a framed 2-tangle $T$ as depicted in Figure \ref{fig:genusone}.

\begin{figure}[htbp]
\includegraphics*[width=90mm]{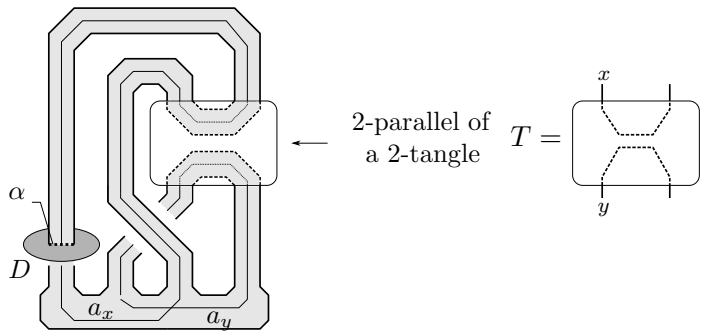}
\begin{picture}(0,0)
\put(-265,20) {$D$}
\put(-265,50) {$\alpha$}
\put(-235,6) {$a_x$}
\put(-190,3) {$a_y$}
\put(-42,95) {\footnotesize $x$}
\put(-42,45) {\footnotesize $y$}
\put(-135,76) {2-parallel of}
\put(-130,64) {a 2-tangle}
\put(-75,70) {\large $T=$}
\end{picture}
\caption{A spine tangle $T$ adapted to the cosmetic crossing} 
\label{fig:genusone}
\end{figure} 

We call the framed tangle $T$ a \emph{spine tangle} of $K$ adapted to the cosmetic crossing of a genus one knot $K$. Let $M= \begin{pmatrix}n  & \ell \\ \ell & m \end{pmatrix}$ be the linking matrix of $T$, where $n$ (resp. $m$) is the framing of the strand $x$ (resp. $y$) and $\ell$ is the linking number of two strands of $T$.

Let $K'$ be a knot obtained from $K$ by crossing change along the crossing disk $D$. Then $S$ gives rise to a Seifert surface $S'$ of $K'$ (\cite[Proposition 2.1]{bfkp}). $K'$ has a spine tangle presentation $T'$, so that $T'$ and $T$ are the same as unframed tangles, and that the linking matrix of $T'$ is $M'= \begin{pmatrix}n \pm 1  & \ell \\ \ell & m \end{pmatrix}$

With respect to the basis $\{a_x,a_y\}$, the Seifert matrix $V$ of $K$ and the Seifert matrix $V'$ of $K'$ are given by
\[ V=\begin{pmatrix} n & \ell \\ \ell \pm 1 & m \end{pmatrix}, \quad V'=\begin{pmatrix} n \pm 1 & \ell \\ \ell \pm 1 & m \end{pmatrix} \]
respectively. Since $K$ and $K'$ are the same knot, 
\[ \Delta_K(t) \doteq \det(V-tV^{T})\doteq  \det(V'-tV'{}^{T}) \doteq  \Delta_{K'}(t).\]
By direct computation, this implies that 
\begin{equation}
\label{eqn:aslice} m=0. 
\end{equation}
In particular, $K$ is algebraically slice.

\section{2-loop polynomial of genus one knot}

Here we quickly review the 2-loop polynomial. For details, see \cite{oht1}.
Let $\mathcal{B}$ be the space of open Jacobi diagram.
For a knot $K$ in $S^{3}$, let $Z^{\sigma}(K) \in \mathcal{B}$ be the Kontsevich invariant of $K$, viewed so that it takes value in $\mathcal{B}$ by composing the inverse of the Poincar\'e-Birkoff-Witt isomorphism $\sigma:  \mathcal{A}(S^{1})\rightarrow \mathcal{B}$.

A Jacobi diagram whose edge is labeled by a power series $f(\hbar)= c_0+c_1\hbar+ c_2\hbar^{2}+c_3\hbar^{3}+\cdots $ represents the Jacobi diagram 
\[ \linef = c_0 \linefa + c_1 \linefb+ c_2\linefc +c_3 \linefd + \cdots\]
It is known that (the logarithm of) the Kontsevich invariant $Z^{\sigma}(K)$ is written in the following form \cite{gk,kri}.
\begin{align*}
\log_{\sqcup} Z^{\sigma}(K) & = \oneloop + \sum_{i:\textrm{finite}} \twoloop\\
& \qquad + ( (\ell>2)\mbox{-loop parts}).
\end{align*}
Here 
\begin{itemize}
\item $\Delta_K(t)$ is the Alexander polynomial of $K$, normalized so that $\Delta_K(1)=1$ and $\Delta_K(t)=\Delta_K(t^{-1})$ hold.
\item $\log_{\sqcup}$ is the logarithm with respect to the disjoint union product $\sqcup$ of $\mathcal{B}$, given by 
\[ \log_{\sqcup}(1+D)=D- \frac{1}{2} D\sqcup D + \frac{1}{3} D\sqcup D \sqcup D + \cdots.\]
\item $p_{i,j}(e^{\hbar})$ is a polynomial of $e^{\hbar}$.
\end{itemize}
Let
\[\Theta(t_1,t_2,t_3;K) = \sum_{\varepsilon \in \{\pm 1\}} \sum_{\sigma \in S_3} p_{i,1}(t^{\varepsilon}_{\sigma(1)})p_{i,2}(t^{\varepsilon}_{\sigma(2)})p_{i,3}(t^{\varepsilon}_{\sigma(3)}).\]
Here $S_3$ is the symmetric group of degree $3$. The \emph{2-loop polynomial} $\Theta_K(t_1,t_2) \in \Q[t_1^{\pm 1},t_2^{\pm 1}]$ of a knot $K$ is defined by
\[ \Theta_K(t_1,t_2)=\Theta(t_1,t_2,t_3;K)|_{t_3=t_1^{-1}t_2^{-1}}.\]

The \emph{reduced 2-loop polynomial} is a reduction of the $2$-loop polynomial defined by 
\[ \widehat{\Theta}_K(t)= \frac{\Theta_K(t,1)}{(t^{\frac{1}{2}}-t^{-\frac{1}{2}})^2} \in \Q[t^{\pm 1}].\]

In general, although Ohtsuki developed fundamental techniques and machineries that enable us to compute $\Theta_K(t_1,t_2)$, the computation of the 2-loop polynomial is much more complicated than the computation of the 1-loop part (i.e., the Alexander polynomial).
Fortunately, when the knot has genus one, Ohtsuki proved a direct formula of $\Theta_K(t_1,t_2)$ \cite[Theorem 3.1]{oht1}. Consequently he gave the following formula of the reduced 2-loop polynomial of genus one knots.
\begin{theorem}\cite[Corollary 3.5]{oht1}
\label{theorem:2-loop}
Let $K$ be a genus one knot expressed by using a framed 2-tangle $T$ as in Figure \ref{fig:genusone}, and let $M=\begin{pmatrix}n & \ell \\ \ell & m \end{pmatrix}$ be the linking matrix of $T$.
Then 
\begin{align*}
\widehat{\Theta}_K(t) & = \Bigl((n+m)(d-\frac{nm}{2})-\ell(\ell+\frac{1}{2})(\ell+1)+12v_3 \Bigr) \Bigl(-2-\frac{2d+1}{3}(t+t^{-1}-2) \Bigr) \\
& \quad -4\left(mv_2^{xx}+nv_2^{yy}-(\ell+\frac{1}{2})v_2^{xy}+3v_3 \right)\Delta_K(t)
\end{align*}
Here 
\begin{itemize}
\item $d=nm-\ell^2-\ell$. In particular, $\Delta_K(t) = dt +(1-2d) + dt^{-1}$.
\item $v_2^{xx},v_2^{yy}, v_2^{xy}$ (resp. $v_3$) are some integer-valued finite type invariant of $T$ whose degree is $2$ (resp. $3$), which do not depend on the framing.
\end{itemize}
\end{theorem}

\section{Constraint for cosmetic crossings}

We prove the Theorem \ref{theorem:main} and Theorem \ref{theorem:main-trivialA} at the same time.

\begin{theorem}
Let $K$ be a genus one knot. If $K$ admits a cosmetic crossing, then $\Delta_K(t)=1$ and $\lambda(\Sigma_2(K))-2w_3(K) \equiv 0 \pmod{16}$.
\end{theorem}

\begin{proof}

Assume that $K$ is a genus one knot admitting a cosmetic crossing. 
We express $K$ using a spine tangle $T$ adapted to the cosmetic crossing. Then as we have seen (\ref{eqn:aslice}), the linking matrix of $T$ is $M=\begin{pmatrix}n & \ell \\ \ell & 0 \end{pmatrix}$.  Moreover, for the knot $K'$ obtained by the crossing change, $K'$ has a spine tangle $T'$ which is identical with $T$ as an unframed tangle with linking matrix is $M'=\begin{pmatrix}n\pm 1 & \ell \\ \ell & 0 \end{pmatrix}$.

Since the finite type invariants $v_2^{xx},v_2^{yy}, v_2^{xy}$ and $v_3$ do not depend on the framing, by Theorem \ref{theorem:2-loop},
\begin{align*}
0 & = \widehat{\Theta}_K(t)-\widehat{\Theta}_{K'}(t)\\
 & = d(-2-\frac{2d+1}{3}(t+t^{-1}-2)) -4v_{2}^{yy}( dt +(1-2d) + dt^{-1})\\
&= d\left(-\frac{2d+1}{3}-4v_2^{yy}\right) t + \frac{d(4d-4)}{3}+4v_2^{yy}(2d-1) + d\left(-\frac{2d+1}{3}-4v_2^{yy} \right)t^{-1}.
\end{align*}
Therefore 
\begin{equation}
\label{eqn:consequence}d\left(-\frac{2d+1}{3}-4v_2^{yy}\right) = \frac{d(4d-4)}{3}+4v_2^{yy}(2d-1)=0.
\end{equation}

If $d\neq 0$, by (\ref{eqn:consequence}) $d=\frac{1}{4}$. Since $d\in\Z$, this is a contradiction so we conclude $d=0$ and $\Delta_K(t)=1$.

Then by (\ref{eqn:consequence}), $d=0$ implies $v_2^{yy}=0$. Moreover, since $d=nm-\ell^2 -\ell =-\ell(\ell+1)$, we get $\ell=0, -1$. 
Thus by Theorem \ref{theorem:2-loop}, the reduced 2-loop polynomial is
\[ \widehat{\Theta}_K(t)=12v_3\left(-2 - \frac{1}{3}(t+t^{-1}-2)\right)-4 \left(-\left(\ell +\frac{1}{2}\right) v_2^{xy}-3v_3\right) \]
hence
\[ \widehat{\Theta}_K(1)=-12v_3 + 4\left( \ell +\frac{1}{2} \right)v_2^{xy}, \  \widehat{\Theta}_K(-1)=4v_3 + 4\left( \ell +\frac{1}{2} \right)v_2^{xy}.\]
On the other hand, by \cite[Proposition 1.1]{oht1}
\[ \widehat{\Theta}_K(1)= 2w_3(K), \quad \widehat{\Theta}_K(-1) =-\frac{1}{12}V'_K(-1)V_K(-1).\]
Since $\Delta_K(-1) = V_K(-1)=1$, by Mullins' formula of the Casson-Walker invariant $\lambda_{w}$  of the double branched coverings \cite{mul},
\[ \lambda_w(\Sigma_2(K)) = -\frac{V'_K(-1)}{6V_K(-1)} + \frac{\sigma(K)}{4}\]
we get
\[ \widehat{\Theta}_K(-1)=\frac{1}{2}\lambda_w(\Sigma_2(K)).\]
For an integral homology sphere, the Casson invariant $\lambda$ is twice of the Casson-Walker invariant $\lambda_{w}$ hence we conclude
\begin{align*}
\lambda(\Sigma_2(K))-2w_3(K) &=\widehat{\Theta}_K(-1)-\widehat{\Theta}_K(1)=16 v_3.
\end{align*}
\end{proof}

\begin{proof}[Proof of Corollary
\ref{corollary:pretzel}]
The reduced $2$-loop polynomial of genus pretzel knots $P(p,q,r)$ was given in \cite[Example 3.6]{oht1}. In particular, for $K=P(4k+1,4k+3,-(2k+1))$, $\widehat{\Theta}_K(1)$ and $\widehat{\Theta}_K(-1)$ are given by 
\[ \widehat{\Theta}_K(1)= -\frac{1}{8}(4k+2)(4k+4)(-2k), \widehat{\Theta}_K(-1) =  -\frac{1}{24}(4k+2)(4k+4)(-2k) \]
hence
\begin{align*}
\lambda(\Sigma_2(K))-2w_3(K)
&=\widehat{\Theta}_K(-1)-\widehat{\Theta}_K(1)=\frac{1}{12}(4k+2)(4k+4)(-2k)\\
& = -16 \frac{(2k+1)(k+1)k}{12}.
\end{align*}
When $k \equiv 1,2 \pmod 4$, $\frac{(2k+1)(k+1)k}{12} \not \in \Z$ hence $K$ does not admit cosmetic crossing by Theorem \ref{theorem:main-trivialA}.
\end{proof}

\section*{Acknowledgement}
The author has been partially supported by JSPS KAKENHI Grant Number 19K03490,16H02145.

\end{document}